\documentclass[12pt,twoside,reqno]{amsart}
\usepackage[latin1]{inputenc}
\usepackage[italian,english]{babel}
\usepackage{amssymb,latexsym}
\usepackage{amsmath,amsfonts,amsthm}
\usepackage{paralist}
\usepackage{color}

\numberwithin{equation}{section}

\newcommand{\R}{\mathbb{R}}
\newcommand{\N}{\mathbb{N}}

\newcommand{\J}{\mathcal{J}}

\newtheorem{lem}{Lemma}

\newtheorem{thm}{Theorem}

\newtheorem{defn}{Definition} 
\theoremstyle{remark}
\newtheorem{remark}{Remark}

\begin{document}
\title[Multiple solutions for a fractional $p$-Laplacian equation]{Multiple solutions for a fractional $p$-Laplacian equation with sign-changing potential}

\author{Vincenzo Ambrosio}
\address{Dipartimento di Matematica e Applicazioni "R. Caccioppoli"\\
         Universit\`a degli Studi di Napoli Federico II\\
         via Cinthia, 80126 Napoli, Italy}
\email{vincenzo.ambrosio2@unina.it}

\maketitle

\begin{abstract}
We use a variant of the fountain Theorem to prove the existence of infinitely many weak solutions for the following
fractional $p$-Laplace equation
\begin{equation*}
(-\Delta)_{p}^{s} u + V(x) |u|^{p-2}u = f(x, u) \, \mbox{ in } \, \R^{N}, 
\end{equation*}
where $s\in (0,1)$, $p\geq 2$, $N\geq 2$, $(- \Delta)_{p}^{s}$ is the fractional $p$-Laplace operator, the nonlinearity $f$ is $p$-superlinear at infinity and the potential $V(x)$ is allowed to be sign-changing.
\end{abstract}

\section{\bf Introduction}
\noindent
In this paper we are interested in the study of the nonlinear fractional $p$-Laplacian equation: 
\begin{equation}\label{P}
(- \Delta)_{p}^{s} u + V(x) |u|^{p-2}u = f(x, u) \, \mbox{ in } \, \R^{N},
\end{equation}
where $s\in (0,1)$, $p\geq 2$ and $N\geq 2$. \\
Here $(- \Delta)_{p}^{s}$ is the fractional $p$-Laplace operator defined, for $u$ smooth enough, by setting
$$
(- \Delta)_{p}^{s}u(x)=2\lim_{\epsilon \rightarrow 0} \int_{\R^{N}\setminus B_{\epsilon}(x)} \frac{|u(x)-u(y)|^{p-2} (u(x)-u(y))}{|x-y|^{N+sp}} dy, \quad x\in \R^{N},
$$
up to some normalization constant depending upon $N$ and $s$.\\
When $p=2$, the equation (\ref{P}) arises in the study of the nonlinear Fractional Schr\"odinger equation
\begin{equation*}
\imath \frac{\partial \psi}{\partial t}+(- \Delta)^{s} \psi=H(x,\psi) \mbox{ in } \R^{N}\times \R
\end{equation*}
when we look for standing wave functions $\displaystyle{\psi(x,t)= u(x) e^{-\imath c t}}$, where $c$ is a constant.
This equation was introduced by Laskin \cite{Laskin1, Laskin2} and comes from an extension of the Feynman path integral from the Brownian-like to the Levy-like quantum mechanical paths.\\ 
Nowadays there are many papers dealing with the nonlinear fractional Schr\"odinger equation: see for instance \cite{A, Cheng, DDPW, FQT, FLS, MBR, Secchi1} and references therein.\\
More recently, a new nonlocal and nonlinear operator was considered, namely the fractional $p$-Laplacian.
In the works of Franzina and Palatucci \cite{FP} and of Lindgren and Linqvist \cite{LL}, the eigenvalue problem associated with $(-\Delta)^{s}_{p}$ is studied, in particular some properties of the first eigenvalue and of the higher order eigenvalues are obtained.
Torres \cite{Torres} established an existence result for the problem (\ref{P}) when $f$ is $p$-superlinear and subcritical. 
Iannizzotto et al. \cite{ILPS} investigated existence and multiplicity of solutions for a class of quasi-linear nonlocal problems involving the $p$-Laplacian operator.\\
Motivated by the above papers, we aim to study the multiplicity of nontrivial weak solutions to (\ref{P}), 
when $f$ is $p$-superlinear and $V(x)$ can change sign. \\
More precisely, we require that the potential $V(x)$ satisfies the following assumptions:
\begin{compactenum}[(V1)]
\item $V \in C(\R^{N})$ is bounded from below; 
\item There exists $r>0$ such that 
$$
\lim_{|y|\rightarrow \infty} |\{x\in \R^{N}: |x-y|\leq r, V(x)\leq M\}|=0
$$
for any $M>0$,
\end{compactenum}
while the nonlinearity $f:\R^{N}\times \R\rightarrow \R$ and its primitive $\displaystyle{F(x, t) = \int_{0}^{t} f(x, z) \, dz}$ are such that
\begin{compactenum}[(f1)]
\item $f\in C(\R^{N}\times \R)$,  and there exist $c_{1}>0$ and $p<\nu<p_{s}^{*}$ such that 
$$
|f(x,t)|\leq c_{1}(|t|^{p-1}+|t|^{\nu-1}) \quad \forall (x,t)\in \R^{N}\times \R,
$$  
where $p^{*}_{s}= \frac{Np}{N-sp}$ if $sp<N$ and  $p^{*}_{s}=\infty$ for $sp\geq N$. 
\item $F(x,0)\equiv 0, F(x,t)\geq 0$ for all $(x,t)\in \R^{N}\times \R$ and 
$$
\lim_{|t| \rightarrow \infty} \frac{F(x,t)}{|t|^{p}}=+\infty \mbox{ uniformly in } x\in \R^{N}.
$$
\item There exists $\theta \geq 1$ such that 
$$
\theta \mathcal{F}(x,t)\geq \mathcal{F}(x,\tau t) \quad \forall (x,t)\in \R^{N}\times \R \mbox{ and } \tau \in [0,1]
$$
where $\displaystyle{\mathcal{F}(x,t)=tf(x,t)-pF(x,t)}$.
\item $f(x,-t)=-f(x,t)$ for all $(x,t)\in \R^{N}\times \R$. 
\end{compactenum}
We recall that the conditions $(V1)$ and $(V2)$ on the potential $V$ and $(f1)$-$(f4)$ with $p=2$ and $s=1$, have been used in \cite{ZX} to extend the well-known multiplicity result due to Bartsch and Wang \cite{BW}.
Examples of $V$ and $f$ satisfying the above assumptions are 
\begin{equation*}
V(x)=
\left\{
\begin{array}{ll}
2n|x|-2n(n-1)+c_{0} & \mbox{ if } n-1\leq |x| \leq (2n-1)/2\\
-2n|x|+2n^{2}+c_{0} & \mbox{ if } (2n-1)/2\leq |x|\leq n \quad (n\in \N \mbox{ and } c_{0}\in \R)
\end{array},
\right.
\end{equation*}
and 
$$
f(x,t)=a(x) |t|^{p-2}t \ln(1+|t|) \quad \forall (x,t)\in \R^{N} \times \R,  
$$
where $a(x)$ is a continuous bounded function with positive lower bound.\\
Our main result can be stated as follows. 
\begin{thm}\label{thm1}
Assume that $f$ satisfies $(f1)-(f4)$ and $V$ satisfies $(V1)-(V2)$.
Then the problem (\ref{P}) has  infinitely many nontrivial weak solutions. 
\end{thm}

\noindent
In order to prove Theorem \ref{thm1}, we will consider the following family of functionals 
\begin{equation*}
\J_{\lambda}(u)= \frac{1}{p} \left[\iint_{\R^{2N}} \frac{|u(x)-u(y)|^{p}}{|x-y|^{N+sp}} dxdy  +\int_{\R^{N}} V(x) |u(x)|^{p} dx \right]- \lambda \int_{\R^{N}} F(x,u) \, dx,
\end{equation*}
with $\lambda \in [1,2]$ and $u\in E$,
where $E$ is the completion of $C^{\infty}_{0}(\R^{N})$ with respect to the norm
$$
||u||_{E}^{p}=\iint_{\R^{2N}} \frac{|u(x)-u(y)|^{p}}{|x-y|^{N+sp}} dxdy+ \int_{\R^{N}} V(x)|u(x)|^{p} dx,
$$
and we will show that $\J_{\lambda}$ satisfies the assumptions of the following variant 
of fountain Theorem due to Zou \cite{Zou}:
\begin{thm}\cite{Zou}\label{FT}
Let $(E, ||\cdot||)$ be a Banach space, $E=\overline{\oplus _{j\in \N} X_{j}}$, with $\dim X_{j}<\infty$ for any $j\in \N$.
Set $Y_{k}=\oplus_{j=1}^{k} X_{j}$ and $Z_{k}=\overline{\oplus_{j=k}^{\infty} X_{j}}$.
Let $\mathcal{J}_{\lambda}: E \rightarrow \R$ a family of $C^{1}(E, \R)$ functionals defined by
$$
\J_{\lambda}(u)=A(u)-\lambda B(u), \quad \lambda \in [1, 2].
$$
Assume that $\J_{\lambda}$ satisfies the following assumptions:
\begin{compactenum}[(i)]
\item $\J_{\lambda}$ maps bounded sets to bounded sets uniformly for $\lambda \in [1, 2]$, $\J_{\lambda}(-u)=\J_{\lambda}(u)$
for all $(\lambda, u)\in [1, 2] \times E$.
\item $B(u) \geq 0$ for all $u\in E$, and $A(u)\rightarrow \infty$ or $B(u)\rightarrow \infty$ as $||u||\rightarrow \infty$.
\item There exists $r_{k}>\rho_{k}$ such that
$$
\beta_{k}(\lambda)=\max_{u\in Y_{k}, ||u||=r_{k}} \J_{\lambda}(u)<
\alpha_{k}(\lambda)=\inf_{u\in Z_{k}, ||u||=\rho_{k}} \J_{\lambda}(u), \quad \forall  \lambda \in [1, 2].
$$
\end{compactenum}
Then
$$
\alpha_{k}(\lambda)\leq \xi_{k}(\lambda)=\inf_{\gamma \in \Gamma_{k}} \max_{u\in B_{k}} \J_{\lambda}(\gamma(u)) \quad \forall \lambda \in [1, 2],
$$
where 
$$
B_{k}=\{u\in Y_{k}: ||u||\leq r_{k} \}
\mbox{ and }
\Gamma_{k}=\{\gamma \in C(B_{k}, X): \gamma \mbox{ is odd }, \gamma=Id \mbox{ on } \partial B_{k} \}.
$$
Moreover, for a.e. $\lambda \in [1,2]$, there exists a sequence $\{u^{k}_{m}(\lambda)\}_{m\in \N}\subset E$ such that
\begin{align*}
\sup_{m\in \N} ||u^{k}_{m}(\lambda)||<\infty, \J'_{\lambda}(u^{k}_{m}(\lambda))\rightarrow 0 \mbox{ and } \J_{\lambda}(u^{k}_{m}(\lambda))\rightarrow \xi_{k}(\lambda) \mbox{ as } m \rightarrow \infty.
\end{align*}
\end{thm}

\begin{remark}
By using $(V1)$ we know that there exists $V_{0}>0$ such that $V_{1}(x)=V(x)+V_{0}\geq 1$ for any $x\in \R^{N}$. Let $f_{1}(x,t)=f(x,t)+V_{0}|t|^{p-2}t$ for all $(x,t) \in \R^{N}\times \R$.
Then it is easy to verify that the study of (\ref{P}) is equivalent to investigate the following problem
\begin{equation*}
(- \Delta)_{p}^{s} u + V_{1}(x) |u|^{p-2}u = f_{1}(x, u) \, \mbox{ in } \, \R^{N}.
\end{equation*}
From now on, we will assume that $V(x)\geq 1$ for any $x\in \R^{N}$ in $(V1)$.
\end{remark}

\section{\bf Preliminaries and functional setting}

\noindent
In this preliminary Section, for the reader's convenience, we recall some basic results related to the fractional Sobolev spaces.
For more details about this topic we refer to \cite{DPV}.

\noindent
Let  $u:\R^{N}\rightarrow \R$ be a measurable function. We say that $u$ belongs to the space $W^{s,p}(\R^{N})$ if $u\in L^{p}(\R^{N})$ and
$$
[u]_{W^{s,p}(\R^{N})}^{p}:=\iint_{\R^{2N}} \frac{|u(x)-u(y)|^{p}}{|x-y|^{N+sp}} dxdy<\infty.
$$
Then $W^{s,p}(\R^{N})$ is a Banach space with respect to the norm
$$
||u||_{W^{s,p}(\R^{N})}=\left[[u]^{p}_{W^{s,p}(\R^{N})} + |u|^{p}_{L^{p}(\R^{N})}\right]^{1/p}. 
$$

\noindent
We recall the main embeddings results for this class of fractional Sobolev spaces:
\begin{thm}\cite{DPV}\label{ce}
Let $s\in (0,1)$ and $p\in [1,\infty)$ be such that $s p<N$.
Then there exists $C=C(N, p, s)>0$ such that
\begin{equation*}
|u|_{L^{p^{*}_{s}}(\R^{N})} \leq C ||u||_{W^{s,p}(\R^{N})}.
\end{equation*}
for any $u\in W^{s,p}(\R^{N})$.
Moreover the embedding $W^{s,p}(\R^{N})\subset L^{q}(\R^{N})$ is locally compact whenever $q<p^{*}_{s}$.\\
If $sp=N$ then $W^{s,p}(\R^{N})\subset L^{q}(\R^{N})$ for any $q\in [p, \infty)$.\\
If $sp>N$ then $W^{s,p}(\R^{N})\subset C^{0,\frac{sp-N}{p}}_{loc}(\R^{N})$.
\end{thm}

\noindent 
Now we give the definition of weak solution for the problem (\ref{P}).\\
Taking into account the presence of the potential $V(x)$, we denote by $E$ the closure of $C^{\infty}_{0}(\R^{N})$  with respect to the norm
$$
||u||:=\left([u]^{p}_{W^{s,p}(\R^{N})}+|u|_{V}^{p} \right)^{1/p}, |u|^{p}_{V}=\int_{\R^{N}} V(x) |u(x)|^{p} dx.
$$
Equivalently
$$
E=\left\{u\in L^{p^{*}_{s}}(\R^{N}): [u]_{W^{s,p}(\R^{N})},|u|_{V}<\infty \right\}.
$$
Let us denote by $(E^{*}, ||\cdot||_{*})$ the dual space of $(E, ||\cdot||)$.\\
We define the nonlinear operator $A: E \rightarrow E^{*}$ by setting
$$
\langle A(u),v \rangle=\iint_{\R^{2N}} \frac{|u(x)-u(y)|^{p-2}(u(x)-u(y))}{|x-y|^{N+sp}}(v(x)-v(y))\, dx dy + \int_{\R^{N}} V(x)|u|^{p-2}uv \, dx,
$$
for $u,v\in E$.
Here $\langle \cdot, \cdot \rangle$ denotes the duality pairing between $E$ and $E^{*}$.\\
Let
$$
B(u)=\int_{\R^{N}} F(x,u) dx
$$
for $u\in E$,
and we set $\J(u)=\frac{1}{p} \langle A(u), u \rangle -B(u)$ for $u\in E$.

\begin{defn}
We say that $u\in E$ is a weak solution to (\ref{P}) if $u$ satisfies
$$
\langle A(u),v \rangle=\langle B'(u),v \rangle
$$
for all $v\in E$. 
\end{defn}

\noindent
Now we show that the following compactness result holds. 
\begin{lem}\label{T1}
Under the assumption $(V1)$ and $(V2)$, the embedding $E\subset L^{q}(\R^{N})$ is compact for any $q\in [p, p^{*}_{s})$.
\end{lem}
\begin{proof}
Let $\{u_{n}\}\subset E$ such that $u_{n}\rightharpoonup 0$ in $E$. We have to show that $u_{n}\rightarrow 0$ in $L^{q}(\R^{N})$ for $q\in [p, p^{*}_{s})$. By the interpolation inequality we only need to consider $q=p$. By using the Theorem \ref{ce} we know that $u_{n} \rightarrow 0$ in $L^{p}_{loc}(\R^{N})$. Thus it suffices to show that, for any $\epsilon>0$, there exists $R>0$ such that 
$$
\int_{B_{R}^{c}(0)} |u_{n}|^{p} \, dx < \epsilon; 
$$
here $B_{R}^{c}(0) = \R^{N} \setminus B_{R}(0)$. \\
Let $\{y_{i}\}_{i\in \N}$ be a sequence of points in $\R^{N}$ satisfying $\R^{N}\subset \bigcup_{i\in \N} B_{r}(y_{i})$ and such that each point $x$ is contained in at most $2^{N}$ such balls $B_{r}(y_{i})$. 
We recall that we are assuming $V(x)\geq 1$ for any $x\in \R^{N}$.\\
Let 
\begin{align*}
A_{R,M} = \{x\in B_{R}^{c} : V(x) \geq M\} \quad \mbox{ and } \quad B_{R,M} = \{x\in B_{R}^{c} : V(x) < M\}. 
\end{align*}
Then 
$$
\int_{A_{R,M}} |u_{n}|^{p} \, dx \leq \frac{1}{M} \int_{\R^{N}} V(x) |u_{n}|^{p} \, dx
$$
and this can be made arbitrarily small by choosing $M$ large.\\
Take $\gamma>1$ such that $p\gamma \leq p^{*}_{s}$ and let $\gamma'= \frac{\gamma}{\gamma-1}$ be the dual exponent of $\gamma$. Then for fixed $M>0$ we have 
\begin{align*}
\int_{B_{R,M}} |u_{n}|^{p} \, dx & \leq \sum_{i\in \N} \int_{B_{R,M}\cap B_{r}(y_{i})} |u_{n}|^{p} \, dx \\
&\leq \sum_{i\in \N} \Bigl(\int_{B_{R,M}\cap B_{r}(y_{i})} |u_{n}|^{p\gamma} \, dx	\Bigr)^{\frac{1}{\gamma}} |B_{R,M}\cap B_{r}(y_{i})|^{\frac{1}{\gamma'}}\\
&\leq \epsilon_{R} \sum_{i\in \N} \Bigl(\int_{B_{R,M}\cap B_{r}(y_{i})} |u_{n}|^{p\gamma} \, dx	\Bigr)^{\frac{1}{\gamma}}\\
&\leq C \epsilon_{R} \sum_{i\in \N} ||u_{n}||^{p}_{W^{s,p}(B_{r}(y_{i}))} \\
&\leq C \epsilon_{R} 2^{N} ||u_{n}||^{p}_{W^{s,p}(\R^{N})}
\end{align*}  
where $\epsilon_{R}=\sup_{y_{i}}  |B_{R,M}\cap B_{r}(y_{i})|^{\frac{1}{\gamma'}}$. By assumption $(V1)$ we can infer that $\epsilon_{R} \rightarrow 0$ as $R \rightarrow \infty$. Then we may make this term small by choosing $R$ large.

\end{proof}

\noindent
Finally we prove the following result which will be fundamental later
\begin{lem}\label{Stype}
If $u_{n} \rightharpoonup u$ in $E$ and $\langle A(u_{n}),u_{n}-u \rangle\rightarrow 0$ then $u_{n} \rightarrow u$ in $E$.
\end{lem}
\begin{proof}
Firstly, let us observe that for any $u, v\in E$
$$
\Bigl|\langle A(u), v \rangle \Bigr|\leq [u]_{W^{s,p}(\R^{N})}^{p-1}[v]_{W^{s,p}(\R^{N})}+|v|_{V}^{p-1}|v|_{V}\leq ||u||^{p-1}||v||.
$$
Then, elementary calculations yield
\begin{align}\label{TV}
0&\leq (||u_{n}||^{p-1}-||u||^{p-1})(||u_{n}||-||u||) \nonumber\\
&\leq ||u_{n}||^{p}-\langle A(u_{n}), u \rangle-\langle A(u), u_{n} \rangle+ ||u||^{p} \nonumber\\
&= \langle A(u_{n}), u_{n} \rangle-\langle A(u_{n}), u \rangle-\langle A(u), u_{n} \rangle+\langle A(u), u \rangle \nonumber\\
&=\langle A(u_{n}), u_{n}-u \rangle-\langle A(u), u_{n}-u \rangle=:I_{n}+II_{n}.
\end{align}
By the hypotheses of lemma follows that $I_{n} , II_{n}\rightarrow 0$ as $n \rightarrow \infty$
so, in view of (\ref{TV}), we have $||u_{n}||\rightarrow ||u||$ as $n \rightarrow \infty$. Since it is well known that the weak convergence and the norm convergence in a uniformly convex space imply the strong convergence, to conclude the proof it will be enough to prove that $E$ is uniformly convex.

\noindent
Fix $\varepsilon\in (0, 2)$ and let $u, v\in E$ such that $||u||, ||v||\leq 1$ and $||u-v||> \varepsilon$.\\
By using the fact that 
$$
\left|\frac{a+b}{2}\right|^{p}+\left|\frac{a-b}{2}\right|^{p}\leq \frac{1}{2}(|a|^{p}+|b|^{p}) \mbox{ for any } a, b\in \R,
$$
follows that
\begin{align*}
&\left|\left|\frac{u+v}{2}\right|\right|^{p}+\left|\left|\frac{u-v}{2}\right|\right|^{p}\\
&=\left\{\left[\frac{u+v}{2}\right]^{p}_{W^{s,p}(\R^{N})}+\left[\frac{u-v}{2}\right]^{p}_{W^{s,p}(\R^{N})}+\left|\frac{u+v}{2}\right|^{p}_{V}+\left|\frac{u-v}{2}\right|^{p}_{V}\right\}\\
&\leq \frac{1}{2} \left([u]^{p}_{W^{s,p}(\R^{N})}+[v]^{p}_{W^{s,p}(\R^{N})}+|u|^{p}_{V}+|v|^{p}_{V}\right) \\
&=\frac{1}{2}(||u||^{p}+||v||^{p})=1,
\end{align*}
which gives that $\displaystyle{\left|\left|\frac{u+v}{2}\right|\right|^{p}< 1-\frac{\varepsilon^{p}}{2^{p}}}$. \\
Choosing 
$$
\delta=1-\Bigl[1-\Bigl(\frac{\varepsilon}{2}\Bigr)^{p} \Bigr]^{\frac{1}{p}}>0,
$$
we can infer that $||\frac{u+v}{2}||< 1-\delta$. Then $E$ is uniformly convex.
\end{proof}

\noindent
Let us introduce the following family of functional $\J_{\lambda}: E \rightarrow \R$ defined by
$$
\J_{\lambda}(u)=\frac{1}{p} \langle A(u),u \rangle-\lambda B(u), \quad \lambda \in [1, 2].
$$
After integrating, we obtain from $(f1)$ that for any $(x,t)\in \R^{N}\times \R$
\begin{equation}\label{F}
|F(x,t)|\leq \frac{c_{1}}{p}|t|^{p}+\frac{c_{1}}{\nu}|t|^{\nu}\leq c_{1}(|t|^{p}+|t|^{\nu}).
\end{equation}
By using $(V1)$-$(V2)$, (\ref{F}) and Lemma \ref{T1} follows that $\J_{\lambda}$ is well defined on $E$.
Moreover $\J_{\lambda}\in C^{1}(E, \R)$, and
\begin{equation}\label{derivative}
\J_{\lambda}'(u)v= \langle A(u), v \rangle-\lambda  \langle B'(u),v \rangle
\end{equation}
where 
$$
\langle B'(u),v \rangle=\int_{\R^{N}} f(x,u) v dx.
$$
Then the critical points of $\J_{1}=\J$ are weak solutions to (\ref{P}).\\
In order to apply the Theorem \ref{FT}, we can observe that $E$ is a separable ($C_{0}^{\infty}(\R^{N})$ is separable and dense in $W^{s,p}(\R^{N})$) and reflexive Banach space, so there exist $(\phi_{n})\subset E$ and   $(\phi^{*}_{n})\subset E^{*}$ such that $E=\overline{span\{\phi_{n}: n\in \N\}}$, $E^{*}=\overline{span\{\phi^{*}_{n}: n\in \N\}}$
and $\langle \phi^{*}_{n}, \phi_{m} \rangle=1$ if $n=m$ and zero otherwise.\\
Then, for any $n\in \N$, we set $X_{n}=span\{\phi_{n}\}$, $Y_{n}=\oplus_{j=1}^{n} X_{j}$ and $Z_{n}=\overline{\oplus_{j=n}^{\infty} X_{j}}$.

\section{Proof of Theorem \ref{thm1}}
\noindent
In this section we give the proof of the main result of this paper.

\noindent
Firstly we prove the following Lemmas:
\begin{lem}\label{lem1}
Assume that $V$ satisfies $(V1)$-$(V2)$ and that $f$ satisfies $(f1)$. Then there exists $k_{1}\in \N$ and a sequence $\rho_{k}\rightarrow \infty$ as $k \rightarrow \infty$ such that 
$$
\alpha_{k}(\lambda)=\inf_{u\in Z_{k}, ||u||=\rho_{k}} \J_{\lambda}(u), \quad \forall k\geq k_{1}.
$$
\end{lem}
\begin{proof}
Let us define
$$
b_{p}(k)=\sup_{u\in Z_{k}, ||u||=1} |u|_{L^{p}(\R^{N})}
$$
and
$$
b_{\nu}(k)=\sup_{u\in Z_{k}, ||u||=1} |u|_{L^{\nu}(\R^{N})}.
$$
We aim to prove that 
\begin{equation}\label{bp}
b_{p}(k)\rightarrow 0 \mbox{ and } b_{\nu}(k)\rightarrow 0 \mbox{ as } k\rightarrow \infty.
\end{equation}
It is clear that $b_{p}(k)$ and $b_{\nu}(k)$ are decreasing with respect to $k$ so there exist $b_{p}, b_{\nu} \geq 0$ such that
$b_{p}(k) \rightarrow b_{p}$ and  $b_{\nu}(k) \rightarrow b_{\nu}$ as $k\rightarrow \infty$.
For any $k\geq 0$, there exists $u_{k}\in Z_{k}$ such that $||u_{k}||=1$ and $|u_{k}|_{p}\geq \frac{b_{p}(k)}{2}$.\\
Taking into account that $E$ is reflexive, we can assume that $u_{k} \rightharpoonup u$ in $E$.
Now, for any $\phi^{*}_{n}\in \{\phi^{*}_{j}\}_{j\in \N}$, we can see that $\langle \phi^{*}_{n}, u_{k} \rangle=0$ for $k>n$, so $\langle \phi^{*}_{n}, u\rangle=\lim_{k \rightarrow \infty} \langle \phi^{*}_{n}, u_{k} \rangle= 0$. Then $\langle \phi^{*}_{n}, u\rangle=0$ for any $\phi^{*}_{n}\in \{\phi^{*}_{j}\}_{j\in \N}$, which gives $u=0$.
Since $E$ is compactly embedded in $L^{p}(\R^{N})$ by Lemma \ref{T1}, we have 
$u_{k}\rightarrow 0$ in $L^{p}(\R^{N})$, which implies that $b_{p}=0$. Similarly we can prove  $b_{\nu}=0$.

\noindent
Then, for any $u\in Z_{k}$ and $\lambda \in [1, 2]$, we can see that 
\begin{align*}
\J_{\lambda}(u)&=\frac{1}{p} \langle A(u), u \rangle -\lambda B(u)\\
&\geq \frac{||u||^{p}}{p}-2\int_{\R^{N}} F(x,u) dx \\
&\geq \frac{||u||^{p}}{p}-2c_{1}(|u|^{p}_{L^{p}(\R^{N})}+|u|_{L^{\nu}(\R^{N})}^{\nu}) \\
&\geq \frac{||u||^{p}}{p}-2c_{1}(b_{p}^{p}(k)||u||^{p}+b_{\nu}^{\nu}(k)||u||^{\nu}).
\end{align*}
By using (\ref{bp}), we can find $k_{1}\in \N$ such that
$$
2c_{1}b_{p}^{p}(k)\leq \frac{1}{2p} \quad \forall k\geq k_{1}.
$$
For each $k\geq k_{1}$, we choose
$$
\rho_{k}:=(8pc_{1}b_{\nu}^{\nu}(k))^{\frac{1}{p-\nu}}.
$$
Let us note that 
\begin{equation}\label{defrho}
\rho_{k}\rightarrow \infty \mbox{ as } k\rightarrow \infty,
\end{equation}
since $\nu>p$.
Then we deduce that 
$$
\alpha_{k}(\lambda):=\inf_{u\in Z_{k}, ||u||=\rho_{k}} \J_{\lambda}(u)\geq \frac{1}{4p}\rho_{k}^{p}>0
$$
for any $k\geq k_{1}$.
\end{proof}

\begin{lem}\label{lem2}
Assume that $(V1), (V2), (f1)$ and $(f2)$ hold. Then for the positive integer $k_{1}$ and the sequence $\rho_{k}$ obtained in Lemma \ref{lem1}, there exists $r_{k}>\rho_{k}$ for any $k\geq k_{1}$ such that
$$
\beta_{k}(\lambda)=\max_{u\in Y_{k}, ||u||=r_{k}} \J_{\lambda}(u)<0.
$$
\end{lem}
\begin{proof}
Firstly we prove that for any finite dimensional subspace $F \subset E$ there exists a constant $\delta>0$ such that
\begin{equation}\label{finiteineq}
|\{x\in \R^{N}: |u(x)|\geq \delta ||u||\}| \geq \delta, \quad \forall u\in F\setminus\{0\}.
\end{equation}
We argue by contradiction and we suppose that for any $n\in \N$ there exists $0\neq u_{n}\in F$ such that
$$
\Bigl|\Bigl\{x\in \R^{N}: |u_{n}(x)|\geq \frac{1}{n} ||u||\Bigr\}\Bigr|<\frac{1}{n}, \quad \forall n\in \N.
$$
Let $v_{n}:=\frac{u_{n}}{||u_{n}||}\in F$ for all $n\in \N$. Then $||v_{n}||=1$ for all $n\in \N$ and 
\begin{equation}\label{2.20}
\Bigl|\Bigl\{x\in \R^{N}: |v_{n}(x)|\geq \frac{1}{n} \Bigr\}\Bigr|<\frac{1}{n}, \quad \forall n\in \N.
\end{equation}
Up to a subsequence, we may assume that $v_{n} \rightarrow v$ in $E$ for some $v\in F$ since $F$ is a finite dimensional space. Clearly $||v||=1$. By using Lemma \ref{T1} and the fact that all norms are equivalent on $F$, we deduce that
\begin{equation}\label{2.21}
|v_{n}-v|_{L^{p}(\R^{N})}\rightarrow 0 \mbox{ as } n \rightarrow \infty.
\end{equation}
Since $v\neq 0$, there exists $\delta_{0}>0$ such that
\begin{equation}\label{2.22}
|\{x\in \R^{N}: |v(x)|\geq \delta_{0} \}|\geq \delta_{0}.
\end{equation}
Set
$$
\Lambda_{0}:=\{x\in \R^{N}: |v(x)|\geq \delta_{0} \}
$$
and for all  $n\in \N$,
$$
\Lambda_{n}:=\Bigl\{x\in \R^{N}: |v_{n}(x)|\geq \frac{1}{n} \Bigr\} \quad \Lambda^{c}_{n}:=\R^{N}\setminus \Lambda_{n}.
$$
Taking into account (\ref{2.20}) and (\ref{2.22}), we get
$$
|\Lambda_{n}\cap \Lambda_{0}|\geq |\Lambda_{0}|-|\Lambda^{c}_{n}|\geq \delta_{0}-\frac{1}{n}\geq \frac{\delta_{0}}{2}.
$$
for $n$ large enough.

\noindent
Therefore we obtain
\begin{align*}
\int_{\R^{N}} |v_{n}-v|^{p} dx &\geq \int_{\Lambda_{n}\cap \Lambda_{0}} |v_{n}-v|^{p} dx\\
&\geq \int_{\Lambda_{n}\cap \Lambda_{0}} (|v|^{p}-|v_{n}|^{p}) dx\\
&\geq \Bigl(\delta_{0}-\frac{1}{n}\Bigr)^{p}|\Lambda_{n}\cap \Lambda_{0}| \\
& \geq \Bigl(\frac{\delta_{0}}{2}\Bigr)^{p+1}>0
\end{align*}
which contradicts (\ref{2.21}).\\
Now, by using the fact that $Y_{k}$ is finite dimensional and (\ref{finiteineq}), we can find $\delta_{k}>0$ such that 
\begin{equation}\label{ter}
|\{x\in \R^{N}: |u(x)|\geq \delta_{k} ||u||\}| \geq \delta_{k}, \quad \forall u\in Y_{k}\setminus\{0\}.
\end{equation}
By $(f2)$, for any $k\in \N$ there exists a constant $R_{k}>0$ such that 
$$
F(x,u)\geq \frac{|u|^{p}}{\delta^{p+1}_{k}} \quad \forall x\in \R^{N} \mbox{ and } |u|\geq R_{k}.
$$
Set 
$$
A_{u}^{k}=\{x\in \R^{N}: |u(x)|\geq \delta_{k}||u|| \}
$$
and let us observe that, by (\ref{ter}), $|A_{u}^{k}|\geq \delta_{k}$ for any $u\in Y_{k}\setminus\{0\}$.

\noindent
Then for any $u\in Y_{k}$ such that $||u||\geq \frac{R_{k}}{\delta_{k}}$, we have
\begin{align*}
\J_{\lambda}(u)&\leq \frac{1}{p}||u||^{p}-\int_{\R^{N}} F(x,u) dx \\
&\leq \frac{1}{p}||u||^{p}-\int_{A_{u}^{k}} \frac{|u|^{p}}{\delta^{p+1}_{k}} dx \\
&\leq  \frac{1}{p}||u||^{p}-||u||^{p}=-\Bigl(\frac{p-1}{p}\Bigr)||u||^{p}.
\end{align*}
Choosing $r_{k}>\max\{\rho_{k}, \frac{R_{k}}{\delta_{k}}\}$ for all $k\geq k_{1}$, follows that
$$
\beta_{k}(\lambda)=\max_{u\in Y_{k}, ||u||=r_{k}} \J_{\lambda}(u)\leq -\Bigl(\frac{p-1}{p}\Bigr)r_{k}^{p}<0, \forall k\geq k_{1}.
$$
\end{proof}

\noindent
By using (\ref{F}) and Lemma \ref{T1} we can see that $\J_{\lambda}$ maps bounded sets to bounded sets uniformly for $\lambda \in [1,2]$. Moreover, by $(f4)$, $\J_{\lambda}$ is even. Then the condition $(i)$ in Theorem \ref{FT} is satisfied.
The condition $(ii)$ is clearly true, while $(iii)$ follows by Lemma \ref{lem1} and Lemma \ref{lem2}.\\
Then, by Theorem \ref{FT},  for any $k\geq k_{1}$ and $\lambda \in [1,2]$ there exists a sequence $\{u^{k}_{m}(\lambda)\}\subset E$ such that
\begin{align*}
\sup_{m\in \N} ||u^{k}_{m}(\lambda)||<\infty, \J'_{\lambda}(u^{k}_{m}(\lambda))\rightarrow 0 \mbox{ and } \J_{\lambda}(u^{k}_{m}(\lambda))\rightarrow \xi_{k}(\lambda) \mbox{ as } m \rightarrow \infty
\end{align*}
where
$$
\xi_{k}(\lambda)=\inf_{\gamma \in \Gamma_{k}} \max_{u\in B_{k}} \J_{\lambda}(\gamma(u))
$$
with
$$
B_{k}=\{u\in Y_{k}: ||u||\leq r_{k} \}
\mbox{ and }
\Gamma_{k}=\{\gamma \in C(B_{k}, X): \gamma \mbox{ is odd }, \gamma=Id \mbox{ on } \partial B_{k} \}.
$$
In particular, from the proof of Lemma \ref{lem1}, we deduce that for any $k \geq k_{1}$ and $\lambda \in [1, 2]$
\begin{equation}\label{2.27}
\frac{1}{4p}\rho_{k}^{p}=:c_{k}\leq\xi_{k}(\lambda)\leq d_{k}:=\max_{u\in B_{k}} \J_{1}(u),
\end{equation}
and $c_{k}\rightarrow \infty$ as $k \rightarrow \infty$ by (\ref{defrho}).
As a consequence, for any $k \geq k_{1}$, we can choose $\lambda_{n} \rightarrow 1$ (depending on $k$) and get the corresponding sequences satisfying  
\begin{align}\label{2.28}
\sup_{m\in \N} ||u^{k}_{m}(\lambda_{n})||<\infty, \J'_{\lambda_{n}}(u^{k}_{m}(\lambda_{n}))\rightarrow 0 \mbox{ as } m \rightarrow \infty.
\end{align}
Now, we prove that for any $k\geq k_{1}$, $\{u^{k}_{m}(\lambda_{n})\}_{m\in \N}$ admits a strongly convergent subsequence $\{u^{k}_{n}\}_{n\in \N}$, and that such subsequence is bounded.
\begin{lem}\label{claim1}
For each $\lambda_{n}$ given above, the sequence $\{u^{k}_{m}(\lambda_{n})\}_{m\in \N}$ has a strong convergent subsequence.
\begin{proof}
By (\ref{2.28}) we may assume, without loss of generality, that as $m \rightarrow \infty$
$$
u^{k}_{m}(\lambda_{n})\rightharpoonup u^{k}_{n} \mbox{ in } E
$$
for some $u^{k}_{n} \in E$.
By Lemma \ref{T1} we have
\begin{equation}\label{conv}
u^{k}_{m}(\lambda_{n})\rightarrow u^{k}_{n} \mbox{ in } L^{p}(\R^{N}) \cap L^{\nu}(\R^{N}). 
\end{equation}
By $(f1)$  and H\"older inequality follows that
\begin{align*}
&\Bigl|\int_{\R^{N}} f(x, u^{k}_{m}(\lambda_{n})) (u^{k}_{m}(\lambda_{n})- u^{k}_{n}) dx \Bigr|\leq \\
&\leq c_{1}|u^{k}_{m}(\lambda_{n})|_{p}^{p-1} |u^{k}_{m}(\lambda_{n})- u^{k}_{n}|_{p}+c_{1}
|u^{k}_{m}(\lambda_{n})|_{\nu}^{\nu-1} |u^{k}_{m}(\lambda_{n})- u^{k}_{n}|_{\nu}
\end{align*}
so, by using (\ref{conv}), we get
\begin{align*}
\lim_{m \rightarrow \infty} \int_{\R^{N}} f(x, u^{k}_{m}(\lambda_{n})) (u^{k}_{m}(\lambda_{n})- u^{k}_{n}) dx=0.
\end{align*}
Since 
$$
 \J'_{\lambda_{n}}(u^{k}_{m}(\lambda_{n})) \rightarrow 0 \mbox{ as } m \rightarrow \infty
$$
and $\langle \J_{\lambda}'(u),v \rangle=\langle A(u),v \rangle-\lambda \langle B'(u),v \rangle$, we deduce that
$$
\langle A(u^{k}_{m}(\lambda_{n})), u^{k}_{m}(\lambda_{n})- u^{k}_{n} \rangle \rightarrow 0 \mbox{ as } m\rightarrow \infty.
$$
Then, by using Lemma \ref{Stype}, we can infer that
$$
u^{k}_{m}(\lambda_{n})\rightarrow u^{k}_{n} \mbox{ in } E \mbox{ as } m\rightarrow \infty.
$$
\end{proof}
\end{lem}

\noindent
Therefore, without loss of generality, we may assume that
$$
\lim_{m \rightarrow \infty} u^{k}_{m}(\lambda_{n})=u^{k}_{n}, \quad \forall n\in \N, k\geq k_{1}.
$$
As a consequence, we obtain
\begin{equation}\label{2.32}
\J'_{\lambda_{n}}(u^{k}_{n})= 0,  \J_{\lambda_{n}}(u^{k}_{n})\in [c_{k}, d_{k}], \quad \forall n\in \N, k \geq k_{1}.
\end{equation}
\begin{lem}\label{claim2}
For any $k\geq k_{1}$, the sequence $\{u^{k}_{n}\}_{n\in \N}$ is bounded.
\begin{proof}
For simplicity we set $u_{n}=u^{k}_{n}$.
We suppose by contradiction that, up to a subsequence
\begin{equation}\label{2.33} 
||u_{n}||\rightarrow \infty  \mbox{ as } n\rightarrow \infty.
\end{equation}
Let $w_{n}=\frac{u_{n}}{||u_{n}||}$ for any $n\in \N$. 
Then, up to subsequence, we may assume that
 \begin{align}\label{2.34}
&w_{n}\rightharpoonup w \mbox{ in } E \nonumber \\
&w_{n}\rightarrow w \mbox{ in } L^{p}(\R^{N})\cap L^{\nu}(\R^{N})\\
&w_{n}\rightarrow w \mbox{ a.e. in } \R^{N}. \nonumber
\end{align}
Now we distinguish two cases.

\noindent
Firstly, let us suppose $w=0$.
As in \cite{J}, we can say that for any $n\in \N$ there exists  $t_{n} \in [0,1]$ such that
\begin{equation}\label{2.36}
\J_{\lambda_{n}}(t_{n}u_{n})=\max_{t\in [0,1]} \J_{\lambda_{n}}(t u_{n}).
\end{equation}
Since (\ref{2.33}) hold true, for any $j\in \N$, we can choose $r_{j}=(2jp)^{1/p} w_{n}$ such that
\begin{equation}\label{2.361}
r_{j}||u_{n}||^{-1}\in (0,1)
\end{equation}
provided $n$ is large enough.
By (\ref{2.34}), $F(\cdot,0)=0$ and the continuity of $F$, we can see 
\begin{equation}\label{2.362}
F(x,r_{j}w_{n})\rightarrow F(x,r_{j}w)=0 \mbox{ a.e. } x\in \R^{N}
\end{equation}
as $n\rightarrow \infty$ for any $j\in \N$. 
Then, taking into account (\ref{F}), (\ref{2.34}), (\ref{2.362}), $(f2)$ and by using the Dominated Convergence Theorem we deduce that 
\begin{equation}\label{2.37}
F(x,r_{j}w_{n})\rightarrow 0 \mbox{ in } L^{1}(\R^{N})
\end{equation}
as $n\rightarrow \infty$ for any $j\in \N$. 
Then (\ref{2.36}), (\ref{2.361}) and (\ref{2.37}) yield
$$
\mathcal{J}_{\lambda_{n}}(t_{n} u_{n})\geq \mathcal{J}_{\lambda_{n}}(r_{j}w_{n})\geq 2j-\lambda_{n} \int_{\R^{N}} F(x,r_{j}w_{n}) dx\geq j
$$
provided $n$ is large enough and for any $j\in \N$. 
As a consequence
\begin{equation}\label{3.31}
\mathcal{J}_{\lambda_{n}}(t_{n}u_{n})\rightarrow \infty \mbox{ as } n\rightarrow \infty.
\end{equation}
Since $\mathcal{J}_{\lambda_{n}}(0)=0$ and $\mathcal{J}_{\lambda_{n}}(u_{n})\in [c_{k}, d_{k}]$, we deduce that $t_{n}\in (0,1)$ for $n$ large enough. Thus, by (\ref{2.36}) we have
\begin{equation}\label{2.38}
\langle{\mathcal{J}'_{\lambda_{n}}(t_{n}u_{n}),t_{n}u_{n}}\rangle=t_{n} \frac{d}{dt}\Bigr|_{t=t_{n}} \mathcal{J}_{\lambda_{n}}(t u_{n})=0.
\end{equation} 
Taking into account $(f3)$, (\ref{2.38}) and (\ref{derivative}) we get
\begin{align*}
\frac{1}{\theta}\mathcal{J}_{\lambda_{n}}(t_{n}u_{n})&=\frac{1}{\theta}\Bigl(\mathcal{J}_{\lambda_{n}}(t_{n}u_{n})- \frac{1}{p}\langle{\mathcal{J}'_{\lambda_{n}}(t_{n}u_{n}),t_{n}u_{n}}\rangle \Bigr)\\
&= \frac{\lambda_{n}}{\theta p} \int_{\R^{N}} \mathcal{F}(x,t_{n}u_{n}) dx \\
&\leq \frac{\lambda_{n}}{p} \int_{\R^{N}} \mathcal{F}(x, u_{n}) dx \\
&= \mathcal{J}_{\lambda_{n}}(u_{n})- \frac{1}{p}\langle{\mathcal{J}'_{\lambda_{n}}(u_{n}), u_{n}}\rangle=\mathcal{J}_{\lambda_{n}}(u_{n})
\end{align*}
which contradicts (\ref{2.32}) and (\ref{3.31}).\\
Secondly, assume that
\begin{equation}\label{zdiv0}
w\not\equiv 0.
\end{equation}
Thus the set $\Omega:=\{x\in \R^{N}: w(x)\neq 0\}$ has positive Lebesgue measure. By using (\ref{2.33}) and (\ref{zdiv0}) we have
\begin{equation}\label{3.34}
|u_{n}(x)|\rightarrow \infty \mbox{ a.e. } x\in \Omega \mbox{ as } n\rightarrow \infty.
\end{equation}
Putting together (\ref{2.34}), (\ref{3.34}), and $(f2)$, and by applying Fatou's Lemma, we can easily deduce that
\begin{align*}
\frac{1}{p}-\frac{\mathcal{J}_{\lambda_{n}}(u_{n})}{||u_{n}||^{p}}
&=\lambda_{n} \int_{\R^{N}}\frac{F(x,u_{n}(x))}{||u_{n}||^{p}} dx \\
&\geq \lambda_{n} \int_{\Omega} |w_{n}|^{p} \frac{F(x,u_{n}(x))}{|u_{n}|^{p}} dx  \rightarrow \infty \mbox{ as }  n\rightarrow \infty
\end{align*}
which gives a contradiction because of (\ref{2.32}).\\
Then, we have proved that the sequence $\{u_{n}\}$ is bounded in $E$.\\
\end{proof}
\end{lem}

\noindent
Now we are able to provide the proof of Theorem \ref{thm1}:
\begin{proof}[Proof of Theorem $1$]
Taking into account Lemma \ref{claim2} and (\ref{2.32}), for each $k\geq k_{1}$, we can use similar arguments to those in the proof of Lemma \ref{claim1}, to show that the sequence $\{u_{n}^{k} \}$ admits a strong convergent subsequence with the limit $u^{k}$ being just a critical point of $\J_{1}=\J$. Clearly, $\J(u^{k})\in [c_{k}, d_{k}]$ for all $k\geq k_{1}$. \\
Since $c_{k}\rightarrow \infty$ as $k\rightarrow \infty$ in (\ref{2.27}), we deduce the existence of infinitely many nontrivial critical points of $\J$. As a consequence, we have that (\ref{P}) possesses infinitely many nontrivial weak solutions.

\end{proof}

\end{document}